\documentclass[reqno,12pt]{amsart}
\usepackage{amsmath,amsfonts,amsthm,amssymb,indentfirst}
\usepackage[all,poly]{xy} 
\usepackage{color}
\usepackage[pagebackref,colorlinks]{hyperref}
\usepackage{kantlipsum}
\usepackage[all,poly]{xy}
\usepackage{mathrsfs}
\usepackage{enumerate}

\theoremstyle{plain}
\newtheorem{lemma}{Lemma}[section]

\newtheorem{proposition}[lemma]{Proposition}

\theoremstyle{definition}
\newtheorem{remark}[lemma]{Remark}
\newtheorem{example}[lemma]{Example}

\newcommand{\Zset}{\mathbb Z}

\newcommand{\M}{\operatorname{\mathbb M}}
\newcommand{\gr}{\operatorname{gr}}

\newcommand{\so}{\mathbf{s}}
\newcommand{\ra}{\mathbf{r}}

\title[Realization of graded matrix algebras as Leavitt path algebras]{Realization of graded matrix algebras as Leavitt path algebras}

\author{Lia Va\v s}
\address{Department of Mathematics, Physics and Statistics, University of the Sciences, Philadelphia, PA 19104, USA}
\email{l.vas@usciences.edu}



\begin{document}

\begin{abstract} 
While every matrix algebra over a field $K$ can be realized as a Leavitt path algebra, this is not the case for every {\em graded} matrix algebra over a graded field. We provide a complete description of graded matrix algebras over a field, trivially graded by the ring of integers, which are graded isomorphic to Leavitt path algebras. As a consequence, we show that there are graded corners of Leavitt path algebras which are not graded isomorphic to Leavitt path algebras. This contrasts a recent result stating that every corner of a Leavitt path algebra of a finite graph is isomorphic to a Leavitt path algebra. If $R$ is a finite direct sum of graded matricial algebras over a trivially graded field and over naturally graded fields of Laurent polynomials, we also present conditions under which $R$ can be realized as a Leavitt path algebra.  
\end{abstract}

\maketitle

\section{Introduction}

Every matrix algebra over a field $K$ or the ring $K[x, x^{-1}]$ is isomorphic to a Leavitt path algebra. In contrast, not every {\em graded} matrix algebra over a field is graded isomorphic to a Leavitt path algebra by \cite[Proposition 3.7]{Lia_no-exit}. Here, a Leavitt path algebra is considered with the natural grading by the ring of integers $\Zset$ and the field $K$ is considered to be trivially $\Zset$-graded. The Leavitt Path Algebra Realization Question of \cite[Section 3.3]{Lia_no-exit} asks for a characterization of those graded matrix algebras over $K$ which can be realized as Leavitt path algebras. In Proposition \ref{matrix_rep_of_sinks}, we answer this question by providing a complete characterization of graded matrix algebras over $K$ which are graded isomorphic to Leavitt path algebras. In Proposition \ref{matrix_rep_of_comets}, we provide analogous characterization for graded matrix algebras over naturally $\Zset$-graded $K[x^m, x^{-m}]$ for any positive integer $m.$ These two results are used in Proposition \ref{matrix_rep_of_no-exits} which presents conditions under which a finite direct sum of graded matricial algebras over $K$ and $K[x^m, x^{-m}]$ can be realized by a Leavitt path algebra.

As a consequence of Proposition \ref{matrix_rep_of_sinks}, we show that there are graded corners of Leavitt path algebras which are not graded isomorphic to Leavitt path algebras (Example \ref{example_on_corner_which_is_not_a_LPA}). This contrasts a recent result from \cite{Gene_Nam} which states that every corner of a Leavitt path algebra of a finite graph is isomorphic to another Leavitt path algebra. 

\section{Prerequisites}

A ring $R$ is graded by a group $\Gamma$ if $R=\bigoplus_{\gamma\in\Gamma} R_\gamma$ for additive subgroups $R_\gamma$ and $R_\gamma R_\delta\subseteq R_{\gamma\delta}$ for all $\gamma,\delta\in\Gamma.$ The elements of the set $H=\bigcup_{\gamma\in\Gamma} R_\gamma$ are said to be homogeneous. The grading is trivial if $R_\gamma=0$ for every nonidentity $\gamma\in \Gamma.$ A graded ring $R$ is a graded division ring if every nonzero homogeneous element has a multiplicative inverse. If a graded division ring $R$ is commutative then $R$ is a graded field. 

We adopt the standard definitions of graded ring homomorphisms and isomorphisms, graded left and right $R$-modules, graded module homomorphisms, and graded algebras as defined in \cite{NvO_book} and \cite{Roozbeh_book}. We use $\cong_{\gr}$ to denote a graded ring isomorphism. 

In \cite{Roozbeh_book}, for a $\Gamma$-graded ring $R$ and $\gamma_1,\dots,\gamma_n\in \Gamma$, $\M_n(R)(\gamma_1,\dots,\gamma_n)$ denotes the ring of matrices $\M_n(R)$ with the $\Gamma$-grading given by  
\begin{center}
$(r_{ij})\in\M_n(R)(\gamma_1,\dots,\gamma_n)_\delta\;\;$ if $\;\;r_{ij}\in R_{\gamma_i^{-1}\delta\gamma_j}$ for $i,j=1,\ldots, n.$ 
\end{center}
The definition of $\M_n(R)(\gamma_1,\dots,\gamma_n)$ in \cite{NvO_book} is different: $\M_n(R)(\gamma_1,\dots,\gamma_n)$ in \cite{NvO_book} corresponds to $\M_n(R)(\gamma_1^{-1},\dots,\gamma_n^{-1})$ in \cite{Roozbeh_book}. More details on the relations between the two definitions
can be found in \cite[Section 1]{Lia_realization}. Although the definition 
from \cite{NvO_book} has been in circulation longer, some matricial representations of Leavitt path algebras involve positive integers instead of negative integers making the definition from \cite{Roozbeh_book} more convenient when working with Leavitt path algebras. Because of this, we opt to use the definition from \cite{Roozbeh_book}. With this definition, if $F$ is the graded free right module $(\gamma_1^{-1})R\oplus \dots \oplus (\gamma_n^{-1})R,$\footnote{
If $M$ is a graded right $R$-module and $\gamma\in\Gamma,$ the $\gamma$-\emph{shifted or $\gamma$-suspended} graded right $R$-module $(\gamma)M$ is defined as the module $M$ with the $\Gamma$-grading given by \[(\gamma)M_\delta = M_{\gamma\delta}\] for all $\delta\in \Gamma.$ 
Any finitely generated graded free right $R$-module is of the form $(\gamma_1)R\oplus\ldots\oplus (\gamma_n)R$ for $\gamma_1, \ldots,\gamma_n\in\Gamma$ and $\operatorname{Hom}_R(F,F)$ is a $\Gamma$-graded ring which is graded isomorphic to $\M_n(R)(\gamma_1,\dots,\gamma_n)$ (both \cite{NvO_book} and \cite{Roozbeh_book} contain details).} then $\operatorname{Hom}_R(F,F)\cong_{\gr} \;\M_n(R)(\gamma_1,\dots,\gamma_n)$ as $\Gamma$-graded rings.

We also recall \cite[Remark 2.10.6]{NvO_book} stating the first two parts in Lemma \ref{lemma_on_shifts} and \cite[Theorem 1.3.3]{Roozbeh_book}  stating part (3) for $\Gamma$ abelian. Although we use these results in case when $\Gamma$ is the ring of integers, we note that the proof \cite[Theorem 1.3.3]{Roozbeh_book} generalizes to arbitrary $\Gamma.$ The last sentence in the lemma is the statement of \cite[Proposition 1.4.4. and Theorem 1.4.5]{Roozbeh_book}. 

\begin{lemma}\cite[Remark 2.10.6]{NvO_book}, \cite[Theorem 1.3.3, Proposition 1.4.4, and Theorem 1.4.5]{Roozbeh_book} 
Let $R$ be a $\Gamma$-graded ring and $\gamma_1,\ldots,\gamma_n\in \Gamma.$ 
\begin{enumerate}
\item If $\pi$ a permutation of the set $\{1,\ldots, n\},$ then 
\begin{center}
$\M_n (R)(\gamma_1, \gamma_2,\ldots, \gamma_n)\;\cong_{\gr}\;\M_n (R)(\gamma_{\pi(1)}, \gamma_{\pi(2)} \ldots, \gamma_{\pi(n)})$
\end{center}
by the map $x\mapsto pxp^{-1}$ where $p$ is the permutation matrix with 1 at the $(i, \pi(i))$-th spot for $i = 1,\ldots,n$ and zeros elsewhere. 
 
\item If $\delta$ in the center of $\Gamma,$ 
\begin{center}
$\;\;\M_n (R)(\gamma_1, \gamma_2, \ldots, \gamma_n)\;=\;\M_n (R)(\gamma_1\delta, \gamma_2\delta,\ldots, \gamma_n\delta).$
\end{center}

\item If $\delta\in\Gamma$ is such that there is an invertible element $u_\delta$ in $R_\delta,$ then  
\begin{center}
$\M_n (R)(\gamma_1, \gamma_2, \ldots, \gamma_n)\;\cong_{\gr}\;\M_n (R)(\gamma_1\delta, \gamma_2\ldots, \gamma_n)$
\end{center}
by the map $x\mapsto u^{-1}xu$ where $u$ is the diagonal matrix with $u_\delta, 1, 1, \ldots, 1$ on the diagonal.
\end{enumerate}

If $\Gamma$ is abelian and $R$ and $S$ are $\Gamma$-graded division rings, then $$\hskip1.3cm\M_n (R)(\gamma_1, \gamma_2, \ldots, \gamma_n)\;\cong_{\gr}\; \M_m (S)(\delta_1, \delta_2, \ldots, \delta_m)$$ implies that $R\cong_{\gr}S,$ that $m=n,$ and the list $\delta_1, \delta_2, \ldots, \delta_m$ is obtained from the list $\gamma_1, \gamma_2, \ldots, \gamma_n$ 
by a composition of finitely many operations as in parts (1) to (3). 
\label{lemma_on_shifts} 
\end{lemma}

To shorten the notation, if each $\gamma_i\in \Gamma, i=1,\ldots, k,$ appears $d_i$ times in the list $$\gamma_1,\gamma_1,\ldots, \gamma_1,\; \gamma_2,\gamma_2\ldots, \gamma_2,\; \ldots\ldots\ldots, \;\gamma_k,\gamma_k,\ldots, \gamma_k,$$ we abbreviate this list as $$d_1(\gamma_1), d_2(\gamma_2),\ldots, d_{k}(\gamma_k).$$ So, if $K$ is a graded field, we use the following abbreviation  
\[\M_{n}(K)(\gamma_1,\gamma_1,\ldots, \gamma_1,\; \gamma_2,\gamma_2\ldots, \gamma_2,\; \ldots\ldots\ldots,\; \gamma_k,\gamma_k,\ldots, \gamma_k)=
\M_{n}(K)(d_1(\gamma_1), d_2(\gamma_2),\ldots, d_{k}(\gamma_k))\] For example, if $K$ is a field trivially graded by the group of integers, we use 
$\M_9(K)(4(0),3(1),2(2))$ to shorten $\M_9(K)(0,0,0,0,1,1,1,2,2).$

\subsection{Leavitt path algebras}
 
Let $E$ be a directed graph. The graph $E$ is row-finite if every vertex emits finitely many edges and it is finite if it has finitely many vertices and edges. A sink of $E$ is a vertex which does not emit edges. A vertex of $E$ is regular if it is not a sink and if it emits finitely many edges. A cycle is a closed path such that different edges in the path have different sources. A cycle has an exit if a vertex on the cycle emits an edge outside of the cycle. The graph $E$ is acyclic if there are no cycles. We say that graph $E$ is no-exit if $v$ emits just one edge for every vertex $v$ of every cycle. 

Let $E^0$ denote the set of vertices, $E^1$ the set of edges and $\so$ and $\ra$ denote the source and range maps of a graph $E.$ If $K$ is any field, the \emph{Leavitt path algebra} $L_K(E)$ of $E$ over $K$ is a free $K$-algebra generated by the set  $E^0\cup E^1\cup\{e^*\ |\ e\in E^1\}$ such that for all vertices $v,w$ and edges $e,f,$

\begin{tabular}{ll}
(V)  $vw =0$ if $v\neq w$ and $vv=v,$ & (E1)  $\so(e)e=e\ra(e)=e,$\\
(E2) $\ra(e)e^*=e^*\so(e)=e^*,$ & (CK1) $e^*f=0$ if $e\neq f$ and $e^*e=\ra(e),$\\
(CK2) $v=\sum_{e\in \so^{-1}(v)} ee^*$ for each regular vertex $v.$ &\\
\end{tabular}

By the first four axioms, every element of $L_K(E)$ can be represented as a sum of the form $\sum_{i=1}^n a_ip_iq_i^*$ for some $n$, paths $p_i$ and $q_i$, and elements $a_i\in K,$ for $i=1,\ldots,n.$ Using this representation, it is direct to see that $L_K(E)$ is a unital ring if and only if $E^0$ is finite in which case the sum of all vertices is the identity. For more details on these basic properties, see \cite{LPA_book}.

A Leavitt path algebra is naturally graded by the group of integers $\Zset$ so that the $n$-component $L_K(E)_n$ is  the $K$-linear span of the elements $pq^*$ for paths $p, q$ with $|p|-|q|=n$ where $|p|$ denotes the length of a path $p.$ While one can grade a Leavitt path algebra by any group $\Gamma$ (see \cite[Section 1.6.1]{Roozbeh_book}), we always consider the natural grading by $\Zset.$

\subsection{Finite no-exit graphs}
\label{subsection_finite_no_exit}
If $K$ is a trivially $\Zset$-graded field, let $K[x^m, x^{-m}]$ be the graded field of Laurent polynomials $\Zset$-graded by $K[x^m, x^{-m}]_{mk}=Kx^{mk}$ and $K[x^m, x^{-m}]_{n}=0$ if $m$ does not divide $n.$  

By \cite[Proposition 5.1]{Roozbeh_Lia_Ultramatricial}, if $E$ is a finite no-exit graph, then $L_K(E)$ is graded isomorphic to 
$$R=\bigoplus_{i=1}^k \M_{k_i} (K)(\gamma_{i1}\ldots,\gamma_{ik_i}) \oplus \bigoplus_{j=1}^n \M_{n_j} (K[x^{m_j},x^{-m_j}])(\delta_{j1}, \ldots, \delta_{jn_j})$$
where $k$ is the number of sinks, $k_i$ is the number of paths ending in the sink indexed by $i$ for $i=1,\ldots, k,$ and $\gamma_{il}$ is the length of the $l$-th path ending in the $i$-th sink for $l=1,\ldots, k_i$ and $i=1,\ldots, k.$ In the second term, $n$ is the number of cycles, $m_j$ is the length of the $j$-th cycle for $j=1,\ldots, n,$ $n_j$ is the number of paths which do not contain the cycle indexed by $j$ and which end in a fixed but arbitrarily chosen vertex of the cycle, and $\delta_{jl}$ is the length of the $l$-th path ending in the fixed vertex of the $j$-th cycle for $l=1,\ldots, n_j$ and $j=1,\ldots, n.$ 

Note that this representation is not necessarily unique as Example \ref{example_finite_comet} shows, but it is unique up to a graded isomorphism. We refer to the graded algebra $R$ above as a {\em graded matricial representation} of $L_K(E).$ 

\begin{example} Consider the graph below. $$
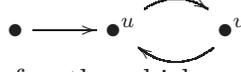
\xymatrix{ {\bullet} \ar[r]& {\bullet}^u \ar@/^1pc/ [r]   & {\bullet}^v \ar@/^1pc/ [l] }$$
If we consider the number and lengths of paths which end at $u,$ we obtain $\M_3(K[x^2, x^{-2}])(0, 1,1)$ as a graded matricial representation of the corresponding Leavitt path algebra. If we consider the paths ending at $v,$ we obtain $\M_3(K[x^2,x^{-2}])(0, 1, 2).$ These two algebras are graded isomorphic by Lemma \ref{lemma_on_shifts} since $(0,1,1)\to (0+1, 1+1, 1+1)\to (1, 2, 2-2)=(1,2,0)\to (0, 1, 2)$ where $\to$ denotes an application of an operation from Lemma \ref{lemma_on_shifts} and results in a graded isomorphism of the corresponding matrix algebras.   
\label{example_finite_comet} 
\end{example}

\section{Realization of graded matrix algebras as Leavitt path algebras}

Every matrix algebra over a field $K$ or the ring $K[x, x^{-1}]$ is isomorphic to a Leavitt path algebra.
Indeed, for any positive integer $n,$ let $L_n$ be the ``line of length $n-1$'', i.e. the graph with $n$ vertices $v_1, v_2,\ldots, v_n$ and an edge from $v_i$ to $v_{i+1}$ for all $i=1, \ldots, n-1.$ Then $L_K(L_n)\cong \M_n(K)$ (\cite[Proposition 1.3.5]{LPA_book} contains more details). 
Adding an edge from $v_n$ to $v_n$ to $L_n$ produces a graph $C_n$ such that $L_K(C_n)\cong \M_n(K[x, x^{-1}]).$ In this section, we provide a complete description of graded matrix algebras over a trivially $\Zset$-graded field $K$ or over the Laurent polynomials $K[x^m, x^{-m}]$ ($\Zset$-graded as in section \ref{subsection_finite_no_exit}) which are graded isomorphic to Leavitt path algebras. As a consequence, we also present conditions under which a finite direct sum of graded matricial algebras over $K$ and over $K[x^m, x^{-m}]$ can be realized as a Leavitt path algebra.

\begin{lemma} Let $n$ and $m$ be positive integers and $\gamma_1,\gamma_2, \ldots, \gamma_n$ be arbitrary integers.
\begin{enumerate}
\item If the smallest element is subtracted from the list $\gamma_1, \gamma_2, \ldots, \gamma_n,$ the elements are permuted so that they are listed in a nondecreasing order, and if $k$ is the largest element of the new list, the new list is $l_0(0), l_1(1), \ldots, l_k(k)$ 
for some nonnegative integers $l_1, \ldots, l_{k-1}$ and some positive $l_0$ and $l_k$ such that $n=\sum_{i=0}^k l_i.$ The integers $k$ and $l_0,l_1, \ldots, l_k$ are unique for the graded isomorphism class of $\M_n(K)(\gamma_1, \gamma_2, \ldots, \gamma_n).$ 
 
\item If the elements $\gamma_1, \gamma_2, \ldots, \gamma_n$ are considered modulo $m$ and arranged in a nondecreasing order, the resulting list is $l_0(0), l_1(1), \ldots, l_{m-1}(m-1)$ for some nonnegative integers $l_0, l_1,\ldots, l_{m-1}$ such that $n=\sum_{i=0}^{m-1}l_i.$ The integers $l_0, l_1,\ldots, l_{m-1}$ are unique for the graded isomorphism class of $\M_n(K[x^m, x^{-m}])(\gamma_1, \gamma_2, \ldots, \gamma_n)$ up to their order.     
\end{enumerate}
\label{representatives} 
\end{lemma}
\begin{proof}
(1) If $k$ and $l_0, l_1,\ldots, l_k$ are obtained as in the statement of part (1), $\M_n(K)(\gamma_1, \gamma_2, \ldots, \gamma_n)\cong_{\gr}$  $\M_n(K)(l_0(0), l_1(1), \ldots, l_k(k))$ by Lemma \ref{lemma_on_shifts}. To show uniqueness, assume that  
\[\M_n(K)(l_0(0), l_1(1), \ldots, l_k(k))\cong_{\gr}\M_n(K)(l'_0(0), l'_1(1), \ldots, l'_{k'}(k'))\] for some nonnegative $k'$ and $ l'_1, \ldots, l'_{k'-1}$ and positive $l_0', l'_{k'}$ such that  $n=\sum_{i=0}^{k'} l'_i.$
By Lemma \ref{lemma_on_shifts}, 
the list $l'_0(0), l'_1(1), \ldots, l'_{k'}(k')$ is obtained from $l_0(0), l_1(1), \ldots, l_k(k)$ by applying finitely many operations of the three types from Lemma \ref{lemma_on_shifts}. Since the 0-component is the only nonzero component of $K,$ the only feasible operation as in part (3) of Lemma \ref{lemma_on_shifts} does not change the list of shifts. If a positive element is added to the list $l_0(0), l_1(1), l_2(2), \ldots, l_k(k),$ the resulting list does not have 0 in it and if a negative element is added to the same list, the resulting list does not consist of nonnegative elements, hence an operation from part (2) of Lemma \ref{lemma_on_shifts} is not present. This means that only an operation from part (1) of Lemma \ref{lemma_on_shifts} can be performed, so $l'_0(0), l'_1(1), \ldots, l'_{k'}(k')$ is obtained by a permutation of $l_0(0), l_1(1), \ldots, l_k(k).$ However, since the elements are already listed in a nondecreasing order, this means that the lists are equal so $k=k'$ and $l_i=l_i'$ for all $i=0,\ldots, k.$  

(2) If $l_0, \ldots, l_{m-1}$ are obtained as in the statement of part (2), $\M_n(K[x^m, x^{-m}])(\gamma_1, \gamma_2, \ldots, \gamma_n)\cong_{\gr}$  $\M_n(K[x^m, x^{-m}])(l_0(0), l_1(1), \ldots, l_{m-1}(m-1))$ by Lemma \ref{lemma_on_shifts}. To show uniqueness, assume that  
\[\M_n(K[x^m, x^{-m}])(l_0(0), l_1(1), \ldots, l_{m-1}(m-1))\cong_{\gr}\M_n(K)(l'_0(0), l'_1(1), \ldots, l'_{m-1}(m-1))\] for some nonnegative $l'_0, l'_1, \ldots, l'_{m-1}$ such that $n=\sum_{i=0}^{m-1} l'_i.$
By Lemma \ref{lemma_on_shifts}, 
the list $l'_0(0), l'_1(1), \ldots,$ $l'_{m-1}(m-1)$ is obtained from $l_0(0), l_1(1), \ldots,l_{m-1}(m-1)$ by applying finitely many operations of the three types from Lemma \ref{lemma_on_shifts}.
Since the elements in both lists of shifts are already in $\{0, 1, \ldots, m-1\},$ if an operation from part (2) of Lemma \ref{lemma_on_shifts} is present, then the results are considered modulo $m$ again using part (3) of Lemma \ref{lemma_on_shifts}. To obtain the resulting list in a nondecreasing order, the elements are permuted using part (1) of Lemma \ref{lemma_on_shifts}. This shows that there is an integer $k$ such that $l'_{i}=l_{i+_m k}$ for all $i=0, \ldots, m-1$ where $+_m$ denotes the operation of the cyclic abelian group $\Zset/m\Zset$ of order $m.$ If we reorder the elements $l_0, \ldots, l_{m-1}$ using the permutation of $\{0, \ldots, m-1\}$ given by $i\mapsto i+_m k,$ the list becomes $l_{0+_mk}=l'_0, \ldots, l_{m-1+_mk}=l'_{m-1}.$
\end{proof}

We say that the nonnegative integers $k$ and $l_0, l_1, \ldots, l_k$ from part (1) of Lemma \ref{representatives} are {\em representatives} of the graded isomorphism class of $\M_n(K)(\gamma_1, \gamma_2, \ldots, \gamma_n).$ By Lemma \ref{representatives}, such representatives are unique. We also say that the nonnegative integers $l_0, l_1, \ldots, l_{m-1}$ from part (2) of Lemma \ref{representatives} are {\em representatives} of the graded isomorphism class of $\M_n(K[x^m, x^{-m}])(\gamma_1, \gamma_2, \ldots, \gamma_n).$ By Lemma \ref{representatives}, such representatives are unique up to their order. For example, $m=2$ and $l_0=1,$ $l_1=2$ for the algebras $\M_3(K[x^2, x^{-2}])(0,1,1)$ and $\M_3(K[x^2,x^{-2}])(0, 1, 2)$ from Example \ref{example_finite_comet}.  

\begin{proposition}
Let $n$ be a positive integer, $\gamma_1,\gamma_2,\ldots, \gamma_n$ be arbitrary integers, and $R$ be the algebra $\M_n(K)(\gamma_1, \gamma_2, \ldots, \gamma_n).$ The following conditions are equivalent. 
\begin{enumerate}
\item $R$ is graded isomorphic to a Leavitt path algebra. 

\item $R$ is graded isomorphic to a Leavitt path algebra of a finite acyclic graph with a unique sink. 
 
\item $R$ is graded isomorphic to $\M_n(K)(0, l_1(1), l_2(2), \ldots, l_k(k))$ for some nonnegative $k$ and positive integers $l_1, \ldots, l_k$ such that $n=1+\sum_{i=1}^k l_i.$    

\item If $k$ and  $l_0, \ldots, l_k$ are representatives of the graded isomorphism class of $R$, then $l_i$ is {\em positive} for all $i=1,\ldots, k$ and $l_0=1.$
\end{enumerate}
\label{matrix_rep_of_sinks}
\end{proposition}
\begin{proof}
If  $R\cong_{\gr} L_K(E)$ for some graph $E$, then $E$ is row-finite and acyclic by \cite[Corollary 3.5]{Lia_no-exit}. Since $R$ is unital, $E$ has finitely many vertices. A row-finite graph with finitely many vertices is finite, so $E$ is finite. 
The algebra $R$ is graded simple (see the second paragraph of \cite[Remark 1.4.8]{Roozbeh_book}), so $E$ has only one sink since otherwise a graded matricial representation of $L_K(E)$ is not graded simple. This shows (1) $\Rightarrow$ (2). The converse (2) $\Rightarrow$ (1) is direct. 

To show (2) $\Rightarrow$ (3), let $R\cong_{\gr} L_K(E)$ for some finite acyclic graph $E$ with a unique sink $v.$ Since the set of lengths of paths of $E$ which end at $v$ is finite, there is a maximal element $k$ of this set and a path $p$ to $v$ of length $k.$ Let $l_i$ be the number of paths of length $i$ to $v$ for $i=0,\ldots, k.$ Then $\M_{n'}(K)(l_0(0), l_1(1), l_2(2), \ldots, l_k(k))$ where $n'=\sum_{i=0}^k l_i$ is graded isomorphic to a graded matricial representation of $L_K(E)$ and, hence, to $R$ as well. The relation $n=n'$ holds by Lemma \ref{lemma_on_shifts}. The trivial path is the only one of length zero so $l_0=1.$ The subpaths of $p$ which end at $v$ have lengths $0, 1,2,\ldots, k,$ so $l_i$ is positive for each $i=0,\ldots, k.$    

To show (3) $\Rightarrow$ (2), let $k$ be any nonnegative integer and $l_1, \ldots, l_k$ be positive integers such that $n=1+\sum_{i=1}^k l_i.$ We construct a finite acyclic graph $E$ with a unique sink such that $L_K(E)\cong_{\gr}\M_n(K)(0, l_1(1), l_2(2), \ldots, l_k(k)).$ Let $E_0$ be an isolated vertex $v_{01}.$ Obtain $E_1$ by adding $l_1$ new vertices $v_{11}, \ldots, v_{1l_1}$ to $E_0$ and an edge from $v_{1j}$ to $v_{01}$ for all $j=1,\ldots, l_1.$ If $E_{i-1}$ is created, obtain $E_i$ by adding $l_i$ new vertices $v_{i1}, \ldots, v_{il_i}$ to $E_{i-1}$ and an edge from $v_{ij}$ to $v_{(i-1)1}$ for all $j=1,\ldots, l_i.$ After $E_k$ is created, let $E=\bigcup_{i=0}^k E_i$. By construction, $E$ is finite and acyclic and $v_{01}$ is the only sink. The trivial path to $v_{01}$ is the only one of length zero and $E$ has exactly $l_i$ paths of length $i$ ending at $v_{01}$ for all $i=1,\ldots, k.$ So, $L_K(E)\cong_{\gr}\M_n(K)(0, l_1(1), l_2(2), \ldots, l_k(k)).$ 

Conditions (3) and (4) are equivalent by Lemma \ref{representatives} since the representatives $k$ and $l_0, \ldots, l_k$ are unique.   
\end{proof}

\begin{remark}
The key requirement in Proposition \ref{matrix_rep_of_sinks} is that the representatives $l_1, \ldots, l_{k-1}$ of the graded isomorphism class of $R$ are {\em positive}. This ensures that there are no ``gaps'' in the lengths of paths. For example, the algebra $\M_2(K)(0, 2)$ is graded isomorphic to no Leavitt path algebra since if there is a path of length 2 to a sink, then there has to be a path of length 1 to that sink also.  
\end{remark}

A graph is said to be a {\em comet} if every vertex connects to a unique cycle of the graph. Such graph is no-exit since if there is an exit $e$ from the only cycle $c,$ then the range of $e$ connects to the cycle $c$ implying the existence of another cycle containing $e$ and a path from the range of $e$ to some vertex of $c.$ Since the cycle $c$ is unique, no such $e$ can exist.

\begin{proposition}
Let $m$ and $n$ be positive integers, $\gamma_1,\gamma_2,\ldots, \gamma_n$ be arbitrary integers, and let $R=\M_n(K[x^m, x^{-m}])(\gamma_1, \gamma_2, \ldots, \gamma_n).$ The following conditions are equivalent.  
\begin{enumerate}
\item $R$ is graded isomorphic to a Leavitt path algebra.
 
\item $R$ is graded isomorphic to a Leavitt path algebra of a finite comet graph.

\item $R$ is graded isomorphic to $\M_n(K[x^m, x^{-m}])(l_0(0), l_1(1), \ldots, l_{m-1}(m-1))$ for some positive integers $l_0, l_1,\ldots, l_{m-1}$ such that $n=\sum_{i=0}^{m-1}l_i.$   

\item If $l_0, \ldots, l_{m-1}$ are representatives of the graded isomorphism class of $R$, then $l_i$ is {\em positive} for all $i=0,\ldots, m-1.$
\end{enumerate}
\label{matrix_rep_of_comets}
\end{proposition}
\begin{proof} 
To show (1) $\Rightarrow$ (2), assume that $R\cong_{\gr} L_K(E)$ for some graph $E$. By \cite[Corollary 3.6]{Lia_no-exit}, $E$ is 
a row-finite no-exit graph without sinks.  Since $R$ is unital, $E$ has finitely many vertices so the condition that $E$ is row-finite implies that $E$ is finite. The algebra $R$ is graded simple, so $E$ has only one cycle since otherwise a graded matricial representation of $L_K(E)$ is not graded simple. Hence, $E$ is a finite comet graph. The converse (2) $\Rightarrow$ (1) is direct. 

To show (2) $\Rightarrow$ (3), let $R\cong_{\gr} L_K(E)$ for some finite comet graph $E.$ If $m'$ is the length of the cycle of $E$, $v$ is a vertex of the cycle, $l_i$ is the number of paths to $v$ of length $i$ modulo $m'$ which do not contain the cycle, and $n'=\sum_{i=0}^{m'-1}l_i,$ then $\M_{n'}(K[x^{m'}, x^{-m'}])(l_0(0), l_1(1), \ldots, l_{m'-1}(m'-1))$ is graded isomorphic to a graded matricial representation of $L_K(E)$ and so to $R$ also. By Lemma \ref{lemma_on_shifts}, $K[x^{m'}, x^{-m'}]\cong_{\gr}K[x^{m}, x^{-m}].$ Assuming that $m'<m,$ produces a contradiction by considering the $m'$-components. One shows that $m\geq m'$ similarly and so $m=m'.$ By Lemma \ref{lemma_on_shifts}, $n=n'.$ For $i=0, \ldots, m-1,$ $l_i$ is positive since there is a subpath of the cycle which ends at $v$ and which has length $i.$ 

To show (3) $\Rightarrow$ (2), consider any positive integers $l_0, \ldots, l_{m-1}$ such that $n=\sum_{i=0}^{m-1}l_i.$ Construct a finite comet graph $E$ as follows. Consider an isolated cycle of length $m$ with vertices $v_0, \ldots, v_{m-1}$ ordered so that $v_{i+1}$ emits an edge to $v_{i} $ for $i=0, \ldots, m-2$ and $v_{0}$ emits an edge to $v_{m-1}.$ For each $i=1, \ldots, m-1,$ add $l_i-1$ new vertices $v_{i1}, \ldots, v_{i(l_i-1)}$ and an edge from $v_{ij}$ to $v_{i-1}$ for each $j=1,\ldots, l_i-1.$ Add also $l_0-1$ new vertices $v_{01}, \ldots, v_{0(l_0-1)}$ and an edge from $v_{0j}$ to $v_{m-1}$ for each $j=1,\ldots, l_0-1.$ The graph $E$ obtained in this way is a finite comet graph with a cycle of length $m.$ For each $i=1, \ldots, m-1,$ there are $l_i-1$ paths to $v_0$ of length $i$ which are not subpaths of the cycle and there is one path from $v_i$ to $v_0$ inside of the cycle. There are $l_0-1$ paths to $v_0$ of length $m$ which are not subpaths of the cycle and there is a trivial path to $v_0.$ So, $l_i$ is the number of paths to $v_0$ of length $i$ modulo $m$ which do not contain the cycle. Thus, $L_K(E)\cong_{\gr}\M_n(K[x^m, x^{-m}])(l_0(0), l_1(1), \ldots, l_{m-1}(m-1)).$

Conditions (3) and (4) are equivalent by Lemma \ref{representatives} since reordering a list of positive elements $l_0, \ldots, l_{m-1}$ produces a list where all elements are also positive. 
\end{proof}

\begin{proposition}
Let $k,n$ be nonnegative, $k_i, n_j, m_j$ positive, and $\gamma_{i1}\ldots,\gamma_{ik_i}, \delta_{j1}, \ldots, \delta_{jn_j}$ arbitrary integers for $i=1,\ldots, k, j=1,\ldots, n.$ If 
$$R=\bigoplus_{i=1}^k \M_{k_i} (K)(\gamma_{i1}\ldots,\gamma_{ik_i}) \oplus \bigoplus_{j=1}^n \M_{n_j} (K[x^{m_j},x^{-m_j}])(\delta_{j1}, \ldots, \delta_{jn_j}),$$ then the following conditions are equivalent. 
\begin{enumerate}
\item $R$ is graded isomorphic to a Leavitt path algebra.
 
\item $R$ is graded isomorphic to a Leavitt path algebra of a finite no-exit graph.

\item There are some nonnegative integers $k'_i$ and positive integers $l_{i1}, l_{i2},\ldots, l_{ik'_i}, i=1,\ldots, k,$ and $s_{j0}, s_{j1},\ldots, s_{j(m_j-1)}, j=1,\ldots,n$ such that $k_i=1+l_{i1}+l_{i2}+\ldots+l_{ik'_i},$ for all $i=1,\ldots, k,$ that $n_j=s_{j0}+s_{j1}+\ldots +s_{j(m_j-1)}$ for all $j=1,\ldots, n,$ and that 
$R$ is graded isomorphic to \[\bigoplus_{i=1}^k \M_{k_i} (K)(0, l_{i1}(1), l_{i2}(2), \ldots, l_{ik'_i}(k'_i)) \oplus \bigoplus_{j=1}^n \M_{n_j} (K[x^{m_j},x^{-m_j}])(s_{j0}(0), s_{j1}(1), \ldots, s_{j(m_j-1)}(m_j-1)).\]    

\item If $k_i'$ and $l_{i0}, \ldots, l_{ik'_i}$ are representatives of the graded isomorphism class of the algebra $\M_{k_i} (K)(\gamma_{i1}\ldots,\gamma_{ik_i})$ for $i=1,\ldots, k$ and if $s_{j0}, \ldots, s_{j(m_j-1)}$ are representatives of the graded isomorphism class of the algebra $\M_{n_j} (K[x^{m_j},x^{-m_j}])(\delta_{j1}, \ldots, \delta_{jn_j})$ for $j=1,\ldots,n$ then $l_{i0}=1$ and $l_{i1}, \ldots, l_{ik'_i}$ are {\em positive} for all $i=1,\ldots, k$ and $s_{j0}, \ldots, s_{j(m_j-1)}$ are {\em positive} for all $j=1,\ldots,n.$ 
\end{enumerate}
\label{matrix_rep_of_no-exits} 
\end{proposition}
\begin{proof}
If $R\cong_{\gr} L_K(E)$ for some graph $E,$ $E$ is 
row-finite and no-exit by  \cite[Corollary 3.4]{Lia_no-exit}.  Since $R$ is unital and $E$ is row-finite, $E$ is finite. This shows (1) $\Rightarrow$ (2). The converse (2) $\Rightarrow$ (1) is direct. 

To show (2) $\Rightarrow$ (3), let $R\cong_{\gr} L_K(E)$ for some finite no-exit graph $E.$ By the graded version of the Wedderburn-Artin Theorem (see \cite[Remark 1.4.8]{Roozbeh_book}), by the argument that $K[x^{m'}, x^{-m'}]\cong_{\gr}K[x^{m}, x^{-m}]$ implies that $m'=m$ shown in the proof of (2) $\Rightarrow$ (3) of Proposition \ref{matrix_rep_of_comets}, and by reordering the terms of $R$ if necessary, we can assume that a graded matricial representation $M$ of $L_K(E)$ is  
$\bigoplus_{i=1}^{k} \M_{k_i} (K)(\gamma'_{i1}\ldots,\gamma'_{ik_i}) \oplus \bigoplus_{j=1}^{n} \M_{n_j} (K[x^{m_j},x^{-m_j}])(\delta'_{j1}, \ldots, \delta'_{jn_j})$ for some integers $\gamma'_{i1}\ldots,\gamma'_{ik_i}$ and $\delta'_{j1}, \ldots, \delta'_{jn_j}.$
For each $i=1,\ldots, k,$ the proof of (2) $\Rightarrow$ (3) in Proposition \ref{matrix_rep_of_sinks} implies that there is a nonnegative integer $k_i'$ and positive integers $l_{i1},\ldots, l_{ik_i'}$ such that $k_i=1+l_{i1}+\ldots+l_{ik'_i}$ and that there is $\phi_i: \M_{k_i} (K)(\gamma'_{i1}\ldots,\gamma'_{ik_i})\cong_{\gr}\M_{k_i} (K)(0, l_{i1}(1), \ldots, l_{ik'_i}(k'_i)).$ 
For each $j=1,\ldots, n,$ the proof of (2) $\Rightarrow$ (3) in Proposition \ref{matrix_rep_of_comets} implies that there are positive integers $s_{j0},\ldots, s_{j(m_j-1)}$ such that $n_j=s_{j0}+\ldots +s_{j(m_j-1)}$ and that there is 
$\psi_j:\M_{n_j} (K[x^{m_j},x^{-m_j}])(\delta'_{j1}, \ldots, \delta'_{jn_j})\cong_{\gr} \M_{n_j} (K[x^{m_j},x^{-m_j}])(s_{j0}(0), \ldots, s_{j(m_j-1)}(m_j-1)).$ If $\phi$ is $\bigoplus_{i=1}^k \phi_i\oplus \bigoplus_{j=1}^n\psi_j,$ then composing $R\cong_{\gr} L_K(E)$ and $ L_K(E)\cong_{\gr}M$ with $\phi$ produces a graded isomorphism of $R$ and a graded algebra as in condition (3).  

To show (3) $\Rightarrow$ (2), let $k'_i$ be a nonnegative integer and let $l_{i1}, \ldots, l_{ik'_i}, s_{j0}, \ldots, s_{j(m_j-1)}$ be positive integers such that $k_i=1+l_{i1}+\ldots+l_{ik'_i}$ and that $n_j=s_{j0}+\ldots +s_{j(m_j-1)}$ for each $i=1,\ldots, k$ and $j=1,\ldots, n.$ By Proposition \ref{matrix_rep_of_sinks}, there is a finite acyclic graph $E_i$ with a unique sink such that $L_K(E_i)\cong_{\gr} \M_{k_i} (K)(0, l_{i1}(1), \ldots, l_{ik'_i}(k'_i))$ for every  $i=1,\ldots k.$ By Proposition \ref{matrix_rep_of_comets}, there is a finite comet graph $F_j$ such that $L_K(F_j)\cong_{\gr}\M_{n_j} (K[x^{m_j},x^{-m_j}])(s_{j0}(0), \ldots, s_{j(m_j-1)}(m_j-1))$ for every $j=1,\ldots,n.$ Let $E$ be the disjoint union of graphs $E_i, i=1,\ldots, k$ and $F_j,$ $j=1,\ldots,n$ so that $L_K(E)$ is graded isomorphic to a graded algebra as in condition (3).

The equivalence of (3) and (4) holds by Lemma \ref{representatives} since representatives of the graded isomorphism class of a matricial algebra over $K$ are unique and representatives of the graded isomorphism class of a matricial algebra over $K[x^m , x^{-m}]$ are unique up to their order. 
\end{proof}

\subsection{Graded corners of Leavitt path algebras} If $R$ is a graded ring and $e$ a homogeneous idempotent, the ring $eRe$ is a {\em graded corner.} By \cite[Theorem 3.15]{Gene_Nam}, every corner of a Leavitt path algebra of a finite graph is isomorphic to another Leavitt path algebra. Using Proposition \ref{matrix_rep_of_sinks}, the example below shows that a {\em graded} corner of a Leavitt path algebra may not be {\em graded} isomorphic to another Leavitt path algebra.  

\begin{example}
Let $E$ be the graph below.  
\[\xymatrix{{\bullet}_u \ar[r]^e& {\bullet}_v \ar[r]^f& {\bullet}_w  } \]
If $\phi$ is the graded isomorphism $L_K(E)\cong_{\gr} \M_3(K)(0,1,2)$ described in section \ref{subsection_finite_no_exit}, then $\phi$ maps the graded idempotent $u+w$ to the graded idempotent $e=e_{11}+e_{33}$ where $e_{11}$ and $e_{33}$ are the standard matrix units. So, the graded corner $e\M_3(K)(0,1,2)e$ is graded isomorphic to the graded algebra $\M_2(K)(0, 2).$ By Proposition \ref{matrix_rep_of_sinks}, $\M_2(K)(0,2)$ is not graded isomorphic to any Leavitt path algebra. 
\label{example_on_corner_which_is_not_a_LPA}
\end{example}

\end{document}